\documentclass[10pt]{article}
\usepackage[utf8]{inputenc}
\usepackage{amsmath,amssymb,amsthm}
\usepackage{mathtools}
\usepackage{enumitem}
\usepackage{xcolor}

\newcommand{\PG}{{\rm PG}}
\newcommand{\LL}{{\cal L}}
\newcommand{\F}{\ensuremath{\mathfrak{F}}}
\newcommand{\C}{{\cal C}}

\newcommand{\FFF}{\ensuremath{\mathcal{F}}}
\newcommand{\qbreuk}[4]{\frac{(q^{#1}-1)\cdots (q^{#2}-1)}{(q^{#3}-1)\cdots (q^{#4}-1)}}

\newtheorem{theorem}{Theorem}[section]
\newtheorem{lemma}[theorem]{Lemma}
\newtheorem{example}[theorem]{Example}
\newtheorem{examples}[theorem]{Examples}

\newtheorem{corollary}[theorem]{Corollary}

\newtheorem*{remark*}{Remark}
\newtheorem*{remarks*}{Remarks}
\newtheorem*{assumption*}{Assumption}

\newcommand{\Gauss}[2]{\ensuremath{\genfrac{[}{]}{0pt}{0}{#1}{#2}}}
\newcommand{\gauss}[2]{\ensuremath{\genfrac{[}{]}{0pt}{1}{#1}{#2}}}

\title{On the chromatic number of two generalized Kneser graphs}

\author{Jozefien D'haeseleer\thanks{Department of Mathematics: Analysis, Logic and Discrete Mathematics, Ghent University, Krijgslaan 281, Building S8, 9000 Gent, Flanders, Belgium}\and Klaus Metsch\thanks{Justus-Liebig-Universit\"at, Mathematisches Institut, Arndtstra\ss e 2, D-35392 Gie\ss en}\and Daniel Werner\footnotemark[2]}

\date{\today}

\begin{document}

\maketitle

\begin{abstract}
We determine the chromatic number of some graphs of flags in buildings of type $A_4$, namely of the Kneser graphs of flags of type $\{2,4\}$ in the vector spaces $GF(q)^5$ for $q\ge3$, and of the Kneser graph of flags of type $\{2,3\}$ in the vector spaces $GF(q)^5$ for large $q$.
\end{abstract}


\section{Introduction}

Let $V$ be an $n$-dimensional vector space over the finite field $GF(q)$ with $n\ge 2$. A \emph{flag} of $V$ is a set $F$ of non-trivial subspaces of $V$ (that is, different from $\{0\}$ and $V$) such that for all $\alpha, \beta \in F$ one has $\alpha \subset \beta$ or $\beta\subset\alpha$. The subset $\{\dim(\alpha)\mid \alpha\in F\}$ is called the \emph{type} of $F$ and it is a subset of $\{1,2,\dots,n-1\}$. Two flags $F$ and $G$ are in \emph{general position} if $\alpha\cap \beta=\{0\}$ or $\alpha+\beta=V$ for all $\alpha\in F$ and $\beta\in G$.

For $\Omega\subseteq \{1,2,\dots,n-1\}$ we define the $q$-Kneser graph $qK_{n;\Omega}$ to be the graph whose vertices are all flags of type $\Omega$ of $V$ with two vertices adjacent when the corresponding flags are in general position and for $k\in\{1,\dots,n-1\}$ we put $qK_{n;k}:=qK_{n;\{k\}}$.

We are interested in the chromatic number of these graphs and hence in their independence number. In many cases these numbers are known, when $|\Omega|=1$, $\Omega=\{k\}$ and $n\ge 2k$. To state the following results, for any $1$-dimensional subspace $T$ of $V$ we call the set of all flags $F$ of type $\Omega$ and for which $F\cup\{T\}$ is a flag, the \emph{point-pencil} (of flags of type $\Omega$) with base point $T$. Furthermore, we use the Gaussian coefficient
\begin{align*}
    {\Gauss ab}_q:=\qbreuk{a}{a-b+1}{b}{}
\end{align*}
for all $a,b\in \mathbb{N}$ and every prime power $q \geq 2$. The Gaussian binomial coefficient ${\gauss ab}_q$ is equal to the number of $b$-spaces of the vector space $\mathbb{F}^a_q$, or in projective context, the number of $(b-1)$-spaces in the projective space $\PG(a-1,q)$.
%If the field size $q$ is clear from the context, we will write ${\gauss ab}$ instead of ${\gauss ab}_q$ and
Moreover, we will denote the number $\gauss{n+1}{1}_q$ by the symbol $\theta_n$.

\begin{theorem}[\cite{Blokhuis_Brouwer_Szoenyi_1}]
If $k > 3$, and either $q > 3$ and $n > 2k + 1$, or $q = 2$ and $n > 2k + 2$, then
the chromatic number of the $q$-Kneser graph is $\chi(qK_{n;k}) = \gauss{n-k+1}{1}_q$.
Moreover, each color class of a minimum coloring is (contained in) a point-pencil and the base points of these point-pencils are the $1$-subspaces of a fixed subspace of dimension $n+1-k$.
\end{theorem}

The main tool in the proof of that result is the knowledge of the independence number of the graph $qK_{n;k}$, the structure of the largest independent sets and a Hilton-Milner type theorem. In this setting, an independent set is a set of pairwise intersecting $k$-spaces in $V$, which is also called an \emph{Ed\H{o}s-Ko-Rado set} in reference to the authors who first introduced the problem in set theory in \cite{erdos_ko_rado}. Indeed, for $n>2k$, a point-pencil is an independent set, and it was shown in \cite{Hsieh1975,frankl} that these are the largest independent sets.

A Hilton-Milner type theorem, named in reference to the authors of \cite{hilton milner}, is a theorem providing the second largest families of pairwise intersecting sets. In the vector space setting such a theorem was proven in \cite{Blokhuis_Brouwer_Szoenyi_1} and states that every independent set, which is sufficiently large, is contained in a point-pencil.

In this paper we determine the chromatic numbers of the graphs $qK_{5;\Omega}$ for  $\Omega=\{2,3\}$ and $q\not=2$ and for $\Omega=\{2,4\}$ and very large $q$. Our first result is the following.

\begin{theorem}\label{main1}
For $q\ge 3$ the chromatic number of the Kneser graph $qK_{5;\{2,4\}}$ is $\gauss{4}{1}_q$. Moreover, each color class of a minimum coloring is contained in a unique point-pencil and the base points of the obtained points-pencil are the $1$-subspaces of a fixed $4$-dimensional subspace.
\end{theorem}

The proof of this result relies on results of \cite{blokhuis2}, where it was not only shown that the independence number of  $qK_{5;\{2,4\}}$ is $(q+1)(q^2+1)(q^2+q+1)$, but also all maximal independent sets with more than $5q^2+5q+1$ elements were determined. One advantage of this graph is that the largest independent sets have an easy structure.

This is different for the graph $qK_{5;\{2,3\}}$. Here point-pencils are independent sets but they are not maximal and maximal independent sets are slightly larger than a point-pencil. Moreover, every point-pencil is contained in more than one largest independent set. Also, since flags of type $\{2,3\}$ are self-dual, there are examples that are dual to the maximal ones that contain a point-pencil. This is partially the reason why we can determine the chromatic number only for very large $q$.

\begin{theorem}\label{main2}
For $q>160\cdot 36^5$ the chromatic number of the Kneser graph $qK_{5;\{2,3\}}$ is $q^3+q^2+1$. Up to duality, each color class of a minimum coloring is related to a unique point-pencil and the $q^3+q^2+1$ corresponding base points are contained in a $4$-space.
\end{theorem}

\begin{remarks*}\rm
\begin{enumerate}
\item In the Kneser graph $qK_{5;\{2,3\}}$ there are various possibilities for a minimum coloring. For example, using point-pencils, the $q^3+q^2+1$ base points corresponding to the color classes are distinct 1-dimensional subspaces of a 4-dimensional subspace $T$. However, the structure of the $q$ remaining 1-dimensional subspaces of $T$ is not fixed, indeed not even the dimension of the subspace they span is fixed.
%For example, the $q^3+q^2+1$ base points corresponding to the color classes are all\textcolor{magenta}{different and contained in a $4$-space $T$. The structure of the $q$ other $1$-spaces in $T$  is not determined, even not the dimension of the subspace of $T$ they span.}
%but $q$ $1$-dimensional subspaces of a $4$-space $T$,  but the structure of these $q$ $1$-spaces is not determined, even not the dimension of the subspace of $T$ they span.
\item We could imagine that our technique for the proof of \ref{main2} can be generalized to determine the chromatic number of the Kneser graphs $qK_{2n+1;\{n,n+1\}}$ for large values of $q$, provided one can determine the independence number and a Hilton-Milner type theorem for these graphs. This has been done however only for $n=2$ as mentioned above in \cite{blokhuis2} and for $n=3$ in \cite{KlausDaniel}.
\item
In \cite{KlausDaniel} the authors classified the largest independent sets of the Kneser graph $qK_{7;\{3,4\}}$ for $q\ge 27$ and also proved an upper bound for the second largest maximal independent sets. We are convinced that similar techniques as used in the present paper can determine the chromatic number of $qK_{6;\{3,4\}}$. Since a proof certainly will be more elaborate, we did not try to integrate such a proof in the present paper. In \cite{pH} the independence number of $qK_{n;\{1,n-1\}}$ was determined and up to our knowledge this has not yet been used to determine the chromatic number of the graph.
\end{enumerate}
\end{remarks*}

This paper is organized as follows. In Section 2 we determine the optimal colorings of the Kneser graph $qK_{5;\{2,4\}}$. In Section \ref{sec_colorings} we provide several examples for optimal colorings of the Kneser graph $qK_{5;\{2,3\}}$. In Section \ref{sec_lemma-on-pointset} we consider three points $P_1,P_2,P_3$ and a set $M$ of points in $\PG(4,q)$, $q$ large, with $M\cap\langle P_1,P_2,P_3\rangle=\emptyset$ and $|M|=cq^3$ for some positive constant $c<1$.

We prove that, if every one of the three points $P_i$ sees $M$ in only few directions, then there exists a solid $S$ that contains at least $mq^2$ points of $M$, where $m$ is a constant. This will be a crucial tool in Section \ref{sec_line-plane}, where we determine the chromatic number of the Kneser graph $qK_{5;\{2,3\}}$ for large values of $q$.

\section{The Kneser-graph $qK_{5;\{2,4\}}$}\label{sec_line-solid}
In this section we work in a vector space $V$ of dimension $5$ over the finite field $GF(q)$ for some prime power $q$. For each $1$-dimensional subspace $P$, we denote by $\FFF(P)$ the point-pencil with base point $P$.

\begin{example}\label{ex24}
If $S$ is a $4$-dimensional subspace of $V$ and $[S]$ the set of all its $1$-dimensional subspaces, then $\{\FFF(P)\mid P\in [S]\}$ is a covering of $qK_{5;\{2,4\}}$ with independent sets.
\end{example}

This example shows that there exists a coloring of $qK_{5;\{2,4\}}$ with $\theta_3$ color classes where each color class is a subset of a point-pencil. Theorem \ref{ddd} below implies that every coloring with at most $\theta_3$ color classes has the same structure of Example \ref{ex24}. For the proof of \ref{ddd}, we use the following result.

\begin{theorem}[\cite{blokhuis2}]\label{blokhuis}
The independence number of $qK_{5;\{2,4\}}$ is $e_0:=\theta_3\theta_2$ and every independent set of $qK_{5;\{2,4\}}$ that is not contained in a point-pencil has at most $e_1:=2q^4+3q^3+4q^2+2q+1$ elements.
\end{theorem}

Our main result in this section is the following theorem.

\begin{theorem}\label{ddd}
Suppose that $\C$ is a covering of the vertices $qK_{5;\{2,4\}}$ consisting of $q^3+q^2+q+1$ maximal independent sets. Then $\C$ consists of all point-pencils with base point contained in a given $4$-dimensional space.
\end{theorem}
\begin{proof}
From Theorem \ref{blokhuis} we have $|F|=e_0$ or $|F|\le e_1$ for each $F\in\C$. Moreover, $|F|=e_0$ implies $F=\FFF(P)$ for some $1$-dimensional subspace $P$. Let $M$ be the set of $1$-dimensional subspaces $P$ of $V$ with $\FFF(P)\in\C$. Let $\LL$ be the set of $2$-dimensional subspaces of $V$ that contain at least one subspace of $M$. For $L\in\LL$ we denote by $c_L$ the number of $1$-dimensional subspaces of $M$ that are contained in $L$. By double counting the flags $(P,L)$ with $P\in M$ and $L\in\LL$ we find
\begin{align*}
\sum_{L\in\LL}c_L=|M|\theta_3,
\end{align*}
since each $1$-dimensional subspace lies in $\theta_3$ subspaces of dimension 2. Next, we double count all triples $(P,P',L)\in M\times M\times\LL$ with $L=P+P'$. Since any two distinct $1$-dimensional subspaces of $M$ generate a $2$-dimensional space, we find
\begin{align*}
\sum_{L\in\LL}c_L(c_L-1)&=|M|(|M|-1).
\end{align*}
For $L\in\LL$ we have $c_L\le q+1$ with equality if all $1$-dimensional subspaces
%\footnote{\textcolor{magenta}{I would prefer $1$-space in stead of $1$-subspace. What do you think?} \textcolor{blue}{I think both "1-space" and "1-subspace" sound terrible and I if it were up to me I would write "1-dimensional subspace" or "subspace of dimension 1" or something like that}} 
of $L$ belong to $M$. It follows that
\begin{align*}
|\LL|=\sum_{L\in\LL}c_L-\sum_{L\in\LL}(c_L-1)\le |M|\theta_3-\frac{|M|(|M|-1)}{q+1}
\end{align*}
with equality if and only if $c_L\in\{1,q+1\}$ for all $L\in\LL$. Since the number of subspaces of dimension $3$ that contain a 2-dimensional subspace is $\theta_2$, the union of all sets $\FFF(P)$ with $P\in M$ contains $|\LL|\theta_2$ flags of type $\{2,4\}$. If we put $x:=\theta_3-|M|$, then $\C$ contains $x$ independent sets of cardinality at most $e_1$ and hence we have
\begin{align}\label{eqn_xbound2}
    \left|\bigcup_{F\in\C}F\right|\le \left(|M|\theta_3-\frac{|M|(|M|-1)}{q+1}\right)\theta_2+xe_1.
\end{align}
Since the union of all independent sets in $\C$ is the set of all  flags of type $\{2,4\}$ and thus has cardinality ${\gauss52}_q\theta_2$, it follows that (use $|M|=\theta_3-x$)
\begin{align}\label{eqn_xbound}
\frac{x^2+xq^4}{q+1}\cdot\theta_2\le xe_1.
\end{align}
First, consider the case when $q\ge 4$. Then (\ref{eqn_xbound}) implies $x=0$ and we have equality in our estimates. Hence $c_L\in\{1,q+1\}$ for all $L\in \LL$. That is, each $L\in \LL$ has the property that either one or all $1$-subspaces of $L$ belong to $M$. This implies that the union of all $1$-subspaces of $M$ is itself a subspace. Since it contains $|M|=q^3+q^2+q+1$ subspaces of dimension $1$, this subspace has dimension $4$ and we are done.

Now, suppose that $q=3$. Then (\ref{eqn_xbound}) shows $x\le 7$ and thus $|M|\ge 33$.
It is not possible that $c_L\le q$ holds for all $L\in\LL$, since otherwise we could improve the bound (\ref{eqn_xbound2}) by replacing $q+1$ in the denominator by $q$ and this improved bound would yield a contradiction. Hence, there exists some $L\in \LL$ with $c_L=q+1=4$. Each of the remaining $|M|-4\ge 29$ elements of
%\footnote{\textcolor{magenta}{Can we also use the word "point" in stead of $1$-subspace? I just do not like that $29$ and $1$ come just after each other. Similar problems occur in the rest of this proof.}\\\textcolor{blue}{I agree that it is not nice when these numbers occur right after one another. However, I would not use "point", because we do not yet use projective language here. Instead I've changed the sentences in question.}}
 $M$ spans a $3$-space with $L$. Since the number of 3-spaces on $L$ is 13, it follows that there exists a {subspace $S$, with $\dim(S)=3$,} (on $L$) that contains at least $4+3=7$ elements of $M$. Similarly, since $|M|\ge 33=26+7$, then one of the four $4$-dimensional subspaces on $S$ contains at least $7+\lfloor \frac{26}{4}\rfloor=14$ elements of $M$. We let $T$ be a $4$-dimensional subspace on $S$ which contains at least $t\ge 14$ elements of $M$. Then the number of subspaces of dimension $2$, that contain one of these $t$ subspaces is at most $130+27t$ with equality only if all $130$ subspaces of dimension $2$ of $T$ belong to $\LL$. If $P\in M$ with $P\not\subseteq T$, then $t$ of the $40$ subspaces of dimension $2$ on $P$ contain an element of $M$ that is contained $T$. It follows that
\[
|\LL|\le 130+27t+(|M|-t)(40-t).
\]
The union of the independent sets $\FFF(P)$ with $P\in M$ is $|\LL|\theta_2$. Since the remaining $x$ independent sets of $\C$ each contain at most $e_1$ flags, and since the total number of $\{2,4\}$ flags is
${\gauss52}_q\theta_2$, it follows that
\begin{align*}
    {\Gauss52}_q\theta_2\le|\LL|\theta_2+xe_1\le(130+27t+(40-x-t)\cdot (40-t))\theta_2+xe_1.
\end{align*}
Using $e_1=(2q^2+q+1)\theta_2$, we can divide by $\theta_2$ and find
\begin{align}\label{eqn_xboundlast}
0\le (t-14) (t+x-39)-4x-26.
\end{align}
Since $14\le t\le |M|=40-x$, it follows first that $t>39-x$, that is $t=|M|=40-x$. Then (\ref{eqn_xboundlast}) gives $0\le -5x$ and hence $x=0$, $t=40$ and $|M|=40$. This implies that $\C$ consists of the sets $\FFF(P)$ for the $40$ subspaces $P$ of dimension 1 of $T$.
\end{proof}

Theorem \ref{main1} follows from Theorem \ref{ddd}.

\section{Colorings of the Kneser-graph $qK_{5;\{2,3\}}$}\label{sec_colorings}

We will now switch to projective language. The vector space $V$ of dimension $5$ over $GF(q)$ and its subspaces correspond to the projective space $\PG(4,q)$ and its subspaces. We remark that a subspace of $V$ of dimension $r$ has projective dimension $r-1$. Subspaces of projective dimension $0$, $1$, $2$ and $3$ will be called \emph{points}, \emph{lines}, \emph{planes} and \emph{solids}, respectively. A flag of type $\{2,3\}$ of $V$ corresponds to a \emph{line-plane flag} of $\PG(4,q)$. Hence, it is a set $\{\ell,\pi\}$ of a line $\ell$ and a plane $\pi$ with $\ell$ contained in $\pi$. Two flags $\{\ell,\pi\}$ and $\{\ell^\prime,\pi^\prime\}$ are adjacent in $qK_{5;\{2,3\}}$ if and only if the flags are in general position in $\PG(4,q)$. This means $l\cap\pi'=\emptyset=l'\cap\pi$ and also implies that $\pi\cap\pi^\prime$ is a point. In projective language, an independent set of the Kneser graph is a set of line-plane flags that are mutually in general position. Here, too, independent sets will be called Erd\H os-Ko-Rado sets of line-plane flags, in short, \emph{EKR-sets}. Thus, the chromatic number of the Kneser graph $qK_{5;\{2,3\}}$ is the smallest number of EKR-sets whose union comprises all line-plane flags.

\textsc{Notation}. Although a flag is a set, we will write flags $\{S,T\}$ of cardinality two of projective spaces as ordered pairs $(S,T)$ where $\dim(S)<\dim(T)$.
\medskip

Point-pencils of line-plane flags are EKR-sets. However, these are not maximal and, as mentioned in the introduction, contained in more than one maximal EKR-set, as we shall see below.

\begin{examples}[EKR-sets]\label{voorbeeld}
Let $\cal M$ be the set of all line-plane flags of $\PG(4,q)$.
For point-line flags $(P,\ell)$, point-solid flags $(P,S)$, plane-solid flags $(\tau,S)$ and point-solid flags $(P,S)$ we define the EKR-sets
\begin{align*}
\FFF(P,\ell)&:=\{(h,\pi)\in {\cal M}\mid P\in h \text{\ or\ }\ell\subset \pi\},\\
\FFF(P,S)&:=\{(h,\pi)\in {\cal M}\mid P\in h \text{\ or\ }P\in \pi\subset S\},\\
\FFF(S,\tau)&:=\{(h,\pi)\in {\cal M}\mid \pi\subset S \text{\ or\ }h\subset \tau\},\\
\FFF(S,P)&:=\{(h,\pi)\in {\cal M}\mid \pi\subset S \text{\ or\ }P\in h\subset S\}.
\end{align*}
Let $F$ be one of the examples above. In the first two cases we call $\FFF(P):=\{(h,\pi)\in {\cal M}\mid P\in h\}$  the \emph{generic part} and $F\setminus \FFF(P)$ the \emph{special part} of $F$. In the remaining two cases, we call $\FFF(S):=\{(h,\pi)\in {\cal M}\mid \pi\subset S\}$ the \emph{generic part} and $F\setminus \FFF(S)$ the \emph{special part} of $F$. 
\end{examples}

%\begin{examples}[of EKR-sets]\label{voorbeeld}
%Let $\cal M$ be the set of all line-plane flags of $\PG(4,q)$.
%\begin{enumerate}
%\item For a point-line flag $(P,\ell)$, we put
%\begin{align*}
%\FFF(P,\ell):=\{(h,\pi)\in {\cal M}\mid P\in h \text{\ or\ }\ell\subseteq \pi\}.
%\end{align*}
%We call $\FFF(P)$ the generic and $\FFF(P,\ell)\setminus\FFF(P)$ the special part of $\FFF(P,\ell)$.
%\item For a point-solid flag $(P,S)$, we put
%\begin{align*}
%\FFF(P,S):=\{(h,\pi)\in {\cal M}\mid P\in h \text{\ or\ }P\in \pi\subseteq S\}.
%\end{align*}
%We call $\FFF(P)$ the generic and $\FFF(P,\ell)\setminus\FFF(P)$ the special part of $\FFF(P,\ell)$.
%\item The structures dual to the two above, that is $\FFF(S,\tau)$ and $\FFF(S,P)$ for a plane-solid flag $(\tau,S)$ or a point-solid flag $(P,S)$. We call the set of all line-plane flags $(\ell,\pi)$ with $\pi\subseteq S$ their generic part and the set of their remaining flags their special part.
%\end{enumerate}
%\end{examples}

Notice that examples 1 and 3 as well as 2 and 4 are dual to each other. Also all four examples have cardinality
\[
e_0:=\Gauss{4}{3}_q\cdot\Gauss{3}{2}_q+q^2\Gauss{3}{1}_q=\theta_2(\theta_3+q^2)
\]
and their special parts have cardinality $q^2\theta_2$. It was shown in \cite{blokhuis1} that these  examples are the largest EKR-sets of line-plane flags in $\PG(4,q)$. We reformulate their result as follows.

 \begin{theorem}[{\cite[Proposition 2.1]{blokhuis1}}]\label{blokuise0e1}
 	Let $\mathcal{F}$ be an EKR-set of line-plane flags of $\PG(4,q)$. Then $|\mathcal{F}|\leq e_0$ and equality occurs if and only if $\mathcal{F}$ is one of the sets defined in Examples \ref{voorbeeld}.
 \end{theorem}

As we explain in the appendix, it was essentially shown in \cite{blokhuis1} that every EKR-set of line-plane flags of $\PG(4,q)$, which is not a subset of one of the sets defined in Examples \ref{voorbeeld}, has cardinality at most
\[
e_1:=4q^4+9q^3+4q^2+q+1.
\]

\begin{examples}[Colorings of $qK_{5;\{2,3\}}$]\label{colorings}
Let $S$ be a solid of $\PG(4,q)$.
\begin{enumerate}[label=\arabic*)]
\item
    Consider a set $W$ of $q$ points of $S$ and suppose that there is a map $\nu$ from $S\setminus W$ to the set of lines of $S$ such that $P\in\nu(P)$ for all $P\in S\setminus W$ and such that every line of $S$ that meets $W$ lies in the image of $\nu$. Then $\{\FFF(P,\nu(P))\mid P\in S\setminus W\}$ is a set of EKR-sets whose union is the set of all line-plane flags of $\PG(4,q)$.

    We provide examples of a set $W$ and a map $\nu$ satisfying these conditions:
    \begin{enumerate}[label=(\alph*)]
    \item\label{example_on_a_line}
        Suppose that $W$ is a set of $q$ points $P_1,\dots,P_q$ which lie on a common line $\ell$ and let $P_0$ be the last remaining point of $\ell$. For each plane $\pi$ of $S$ on $\ell$, fix a numbering $\ell_1(\pi),\dots,\ell_q(\pi)$ of the lines $\not=\ell$ of $\pi$ on $P_0$. Define the map $\nu$ from the set $S\setminus W$ to the line-set of $S$ by $\nu(P_0)=\ell$ and $\nu(P):=PP_i$, if $P\notin\ell$ and $P\in \ell_i(\langle P,\ell \rangle)$.
    \item
        Suppose that $W$ is a set of $q$ points in a plane $\pi$. If such a map $\nu$ is defined on $\pi\setminus W$, then one can extend it to $S$ as follows: The $q$ points in $W$ meet at most $q(q+1)$ lines of $\pi$ and thus there is at least one line $g\le\pi$ which does not meet the set $W$. Let $\pi_1,\dots,\pi_q$ be the planes on $g$ in $S$ different from $\pi$ and for all $i\in\{1,\dots,q\}$ and all $P\in\pi_i\setminus\pi$ set $\nu(P):=PP_i$. % This is indeed a map as required, because for all $i\in\{1,\dots,q\}$ every line $\ell$ with $P_i\le\ell\le S$ which does not lie in $\pi$ meets the plane $\pi_i$ in a point not in $\pi$ and thus is the line $\nu(\ell\cap\pi_i)$.

        Obviously, one can define such a map $\nu$ on a plane $\pi\setminus W$ if $W$ only spans a line therein, because then the construction in \ref{example_on_a_line} can be used. However, one can also find such a map $\nu$ if $W$ spans the plane $\pi$ and we give a simple construction in the case where $q-1$ points $P_1,\dots,P_{q-1}$ of $W$ lie on a common line $\ell_0$ and the last point $P_q$ of $W$ satisfies $\pi=\langle P_q,\ell_0\rangle$. We let $Q_0$ and $Q_1$ be the two remaining points of $\ell_0$ and we fix a numbering $\ell_1,\dots,\ell_q$ of the lines different from $\ell_0$ of $\pi$ on $Q_0$, such that $\ell_q=Q_0P_q$. Then, for all $i\in\{1,\dots,q-1\}$ and all $P\in\ell_i\setminus\{Q_0\}$ we set $\nu(P):=PP_i$. Furthermore, we set $\nu(Q_0):=\ell_0$, $\nu(Q_1):=Q_1P_q$ and for all $P\in\ell_q\setminus\{Q_0,P_q\}$ we set $\nu(P):=\ell_q$. % Then $\nu$ is indeed a map as required, because for all $i\in\{1,\dots,q-1\}$ every line on $P_i$ meets the line $\ell_i$ in a point $P$ and we have set $\nu(P)=PP_i$ (note that this is also true for $P=Q_0$) and every line on $P_q$ either meets $\ell_0$ in $Q_0$, that is, it is the line $\ell_q=\nu(P)$ for some every point $P\in\ell_q\setminus\{Q_0,P_q\}$, or it meets it in the point $Q_1$ and is the line $P_qQ_1=\nu(Q_1)$, or it meets it $\ell_0$ in a point of $W\setminus\{P_q\}$ and all lines on these points are already covered.
    \end{enumerate}
\item
    Finally, we give an example which also uses EKR-sets with special part coming from a solid, that is, EKR-sets $\FFF(P,S)$ for a point-solid flag $(P,S)$.

    Here, suppose again that the $q$ points $P_1,\dots,P_q$ of $W$ only span a line $\ell$ of $S$ and let $P_0$ be the last remaining point of $\ell$. For any plane $\pi$ with $\ell\le\pi\le S$ fix a numbering $\ell_1(\pi),\dots,\ell_q(\pi)$ of the lines of $\pi$ on $P_0$ different from $\ell$  as well as a numbering $S_1(\pi),\dots,S_q(\pi)$ of the solids on $\pi$ different from $S$ and put
    \begin{align*}
        \F_1(\pi)&:=\bigcup_{i=1}^q\{\FFF(P,PP_i)\mid P_0\not=P\in\ell_i(\pi)\},\\
        \F_2(\pi)&:=\bigcup_{i=1}^q\{\FFF(P,S_i(\pi))\mid P_0\not=P\in\ell_i(\pi)\}.
    \end{align*}
    Now, let $\Pi$ be the set consisting of all planes of $S$ that contain $\ell$ and for every subset $R$ of $\Pi$ put
    \[
        \F(R):=\{\FFF(P_0,\ell)\}\cup\bigcup_{\mathclap{\pi\in R}}\F_1(\pi)\cup\bigcup_{\mathclap{\pi\in \Pi\setminus R}}\F_2(\pi).
    \]
    Then for all $R\subseteq\Pi$ the set $\F(R)$ consists of $\theta_3-q$ EKR-sets whose union is the set of all line-plane flags.
\end{enumerate}
\end{examples}

This list of examples is not a complete list of all colorings with $\theta_3-q$ colors. For example, one can also find colorings by replacing the EKR-sets in the colorings described above by their dual structure. However, this list shows that the chromatic number of $\Gamma$ is at most $\theta_3-q$ and provides several examples of different structures. We will prove in Section \ref{sec_line-plane} that the chromatic number is in fact equal to $\theta_3-q$, provided $q$ is large enough.

\section{A lemma on point sets}\label{sec_lemma-on-pointset}
\begin{lemma}\label{solid1}
 Suppose that $M$ is a set of points in $\PG(4,q)$, and $P_1,P_2,P_3$ are three non-collinear points such that the plane $\pi$ they span has no point in $M$. Let $m$, $n$ and $d$ be positive real numbers such that the following hold:
\begin{itemize}
\item Each of the points $P_1,P_2,P_3$ lies on at most $nq^2$ lines that meet $M$,
\item $|M|=dq^3$,
\item $q>32n^5m/d^5$.
\end{itemize}

Then there exists a solid $S$ on $\pi$ with $|S\cap M|\ge mq^2$.
\end{lemma}
\begin{proof}
Let $\pi_j$, $1\le j\le q^2+q$ be the planes on the line $P_1P_2$ different from $\pi$, and for $i\in\{1,2\}$ and $j\in\{1,\dots,q^2+q\}$ let $a_{ij}$ be the number of lines of $\pi_j$ on $P_i$ that meet $M$. Then $x_j:=|\pi_j\cap M|\le a_{1j}a_{2j}$. This implies that $\sqrt{x_j}\le \frac{1}{2}(a_{1j}+a_{2j})$. Since each of $P_1$ and $P_2$ lies on at most $nq^2$ lines that meet $M$, it follows that
\begin{align*}
nq^2\ge\frac{1}{2}\sum_j(a_{1j}+a_{2j})\ge \sum_j\sqrt{x_j}.
\end{align*}
Put $R:=\{j\mid x_j\ge c q^2\}$ with $c=\frac{d^2}{4n^2}$. Then
\begin{align*}
nq^2\ge\sum_{j\notin R}\sqrt{x_j}\ge \frac{1}{\sqrt{c}q}\sum_{j\notin R}x_j\ge \frac{1}{\sqrt{c}q}(|M|-|R|q^2),
\end{align*}
since the sum of $x_j$ over all $j$ is $|M|$ and since each plane $\pi_j$ with $j\in R$ meets $M$ in at most $q^2$ points. It follows that
\begin{align*}
|R|q^2\ge |M|-nq^2\sqrt{c}q
\end{align*}
Assume the statement is not true, that is, every solid on $\pi$ meets $M$ in at most $mq^2$ points. Then every solid on $\pi$ contains at most $mq^2/cq^2$ planes $\pi_j$ with $j\in R$. Hence, the number of solids on $\pi$ that contain a plane $\pi_j$ with $j\in R$ is at least $|R|c/m$. This implies that $P_3$ lies on at least
\[
\frac{|R|c}{m}\cdot cq^2
\]
lines that meet $M$. Hence
\[
\frac{|R|c}{m}\cdot cq^2\le nq^2.
\]
Comparing this to the lower bound for $|R|$, we find
\[
(|M|-nq^2\sqrt{c}q)\frac{c}{m}\cdot c\le nq^2
\ \Longrightarrow\ |M|\le \sqrt{c}nq^3+\frac{mnq^2}{c^2}.
\]
Since $\sqrt{c}n=\frac{1}{2}d$, we find a contradiction to the hypotheses.
\end{proof}

\begin{remark*}
This lemma is the main reason why we can prove Theorem \ref{main2} only for very large values of $q$. In partiular, we will only need a very large value of $q$ when we apply this lemma, but the remaining arguments in the next section only require $q>379$.
\end{remark*}

\section{The chromatic number of $qK_{5;\{2,3\}}$}\label{sec_line-plane}
In this section we prove, for large values of $q$, that the chromatic number of $qK_{5;\{2,3\}}$ is $\theta_3-q$. Notice that we already know a coloring with this many colors, so we only have to show that one can not do better.
\medskip

\textsc{Notation}. We assume that $\F=\{F_1,\dots,F_{\theta_3-q}\}$ is a family of $\theta_3-q$ distinct non-empty EKR-sets of line-plane flags of $\PG(4,q)$ whose union consists of all line-plane flags of $\PG(4,q)$. We shall show that each element $F$ of $\F$ is necessary in the sense that there will be a flag that lies in $F$ but in no other member of $\F$. This implies that the chromatic number of the graph $qK_{5;\{2,3\}}$ is at least $\theta_3-q$, and hence it is exactly $\theta_3-q$.

We may (and will) assume that all EKR-sets of $\F$ are maximal. This implies that $|F|=e_0$ or $|F|\le e_1$ for all $F\in\F$, see Section \ref{sec_colorings}. We may also assume that no two members of $\F$ of size $e_0$ have the same generic part, since otherwise we replace one of the two EKR-sets by its special part, which has cardinality $q^2\theta_2<e_1$ (if this special part happens to already be a member of $\F$, then we can still add up to $e_1-q^2\theta_2$ flags to avoid this). According to Theorem \ref{blokuise0e1} every $F\in\F$ with $|F|=e_0$ is either based on a point $P$ and thus $F$ is equal to $\FFF(P,\ell)$ or $\FFF(P,S)$ for a line $\ell$ or a solid $S$ on $P$, or based on a hyperplane $S$ and thus is equal to $\FFF(S,\pi)$ and $\FFF(S,P)$ for a plane $\pi$ or a point $P$ of $S$. By passing to the dual space if necessary, we may assume that at least one half of these are based on a point. By $I$ we denote the set consisting of all $i\in \{1,\hdots,\theta_3-q\}$ such that $|F_i|=e_0$ and $F_i$ is based on a point. For $i\in I$ we denote the base point of $F_i$ by $P_i$, that is, $F_i=\FFF(P_i,X)$ with $X$ a line or a solid.

\begin{lemma}\label{delta}
The number of all line-plane flags of $\PG(4,q)$ is equal to
\begin{align*}
\Gauss{5}{3}_q\cdot\Gauss{3}{2}_q=|\F|e_0-(2q^7+3q^6+4q^5+3q^4+2q^3+q^2).
\end{align*}
\end{lemma}

\begin{lemma}\label{solid2}
Suppose that $q\ge 41$ and let $S$ be a solid. Denote by $c_1$ the number of indices $i\in I$ with $P_i\notin S$ and by $c_3$ the number of EKR-sets $F\in\F$ with $|F|\le e_1$. Then $(|I|-c_1)+c_3\le 5q^2$ or $c_1+c_3\le 4q^2$.
\end{lemma}
\begin{proof}
We have $|I|\ge\frac{1}{2}(|\F|-c_3)$, we know that for all $i\in I$ the set $F_i$ is based on a point $P_i$ and we set $A:=\{a\in I\mid P_a\in S\}$. For $a\in A$ the set $F_a$ contains $\theta_2^2$ flags $(\ell,\pi)$ with $P\in \ell\subseteq S$. Since there are $(q^2+1)\theta_2$ lines in $S$, there are at most $(q^2+1)\theta_2^2$ flags $(\ell,\pi)$ with $\ell\subseteq S$. It follows that
\begin{align*}
\left|\bigcup_{a\in A}F_a\right|\le |A|(e_0-\theta_2^2)+(q^2+1)\theta_2^2.
\end{align*}
If $i\in I\setminus A$, then for each $a\in A$, the sets $F_i$ and $F_a$ share the $\theta_2$ line-plane flags $(\ell,\pi)$ with $\ell=P_iP_a$. Different $a$ in $A$ yield different $\theta_2$ pairs, and hence $F_i$ contains at least $|A|\theta_2$ flags that are contained in $\cup_{a\in A}F_a$. It follows that
\begin{align*}
\left|\bigcup_{i\in I}F_i\setminus\bigcup_{a\in A}F_a\right|\le|I\setminus A|e_0 -|A||I\setminus A|\theta_2.
\end{align*}
Lemma \ref{delta} implies therefore that
\begin{align*}
&c_3(e_0-e_1)+|A|\theta_2^2-(q^2+1)\theta_2^2+|A|(|I|-|A|)\theta_2
\\ & \le 2q^7+3q^6+4q^5+3q^4+2q^3+q^2=q^2\theta_2(2q^3+q^2+q+1).
\end{align*}
Since $e_0-e_1\ge \theta_2(q^3-3q^2-4q+5)$ and $|A|=|I|-c_1$, this implies
\begin{align}\label{diff1}
&c_3(q^3-3q^2-4q+5)+(|I|-c_1)\theta_2+c_1(|I|-c_1)  \le 2q^5+(2q^2+1)\theta_2.
\end{align}
Assume the statement of the lemma is wrong. Then
\begin{align}\label{diff2}
0\le (c_1+c_3-4q^2)(|I|-c_1+c_3-5q^2).
\end{align}
If we add the right hand side of (\ref{diff2}) to the right hand side of (\ref{diff1}) we find a valid inequality. If we replace in this inequality $|I|$ by $\frac{1}{2}(\theta_3-q-c_3)+z$ with $z:=|I|-\frac{1}{2}(\theta_3-q-c_3)$ we find after simplifications
\begin{align}\label{diff3}
\begin{multlined}
    2(5q^2+q+1-c_3)z+(q^3+6q^2-9q+8-c_3)c_3+q^5\\
        \le 38q^4+2q^3+q+1+(2q+2)c_1.
\end{multlined}
\end{align}
Now, we have $z\ge 0$, (\ref{diff1}) implies $c_3\le 3q^2$ for $q\ge 13$ and hence $(q^3+6q^2-9q+8-c_3)c_3\ge 0$ as well as $2(5q^2+q+1-c_3)z\ge 0$, so (\ref{diff3}) implies
\begin{align*}
q^5 \le 38q^4+2q^3+q+1+(2q+2)c_1.
\end{align*}
As $c_1\le |I|\le|\F|\le \theta_3-q$, this is a contradiction for $q\ge 41$.
\end{proof}

\begin{assumption*}
For the remaining part of this section we assume that
\begin{align}\label{eqn_q}
q>160\cdot 36^5.
\end{align}
\end{assumption*}

\begin{lemma}\label{solid3}
There exists a solid $S$ such that
\[
|\{F\in\F:|F|\le e_1\}|+|\{i\in I: P_i\notin S\}|\le 4q^2.
\]
 \end{lemma}
\begin{proof}
Let $c_3$ be the number of $F\in\F$ with $|F|\not=e_0$ and thus $|F|\le e_1$. Then $\F$ contains $\beta\ge\frac{1}{2}(\theta_3-q-c_3)$ EKR-sets that are largest EKR-sets based on a point. Let these be $G_i$, $i=1,\dots,\beta$, let $R_i$ be the base point of $G_i$ and put
\begin{align*}
g_i:=\left|G_i\cap\bigcup_{j=1}^{i-1} G_j\right|.
\end{align*}
Then $|\cup G_i|=\beta e_0-\sum_{i=1}^\beta g_i$. We may assume that the sequence $g_1,\dots,g_\beta$ is monotone increasing. We want to show that $g_j$ for $j:=\frac14q^3+q^2+2q+1$ is less than $9q^2\theta_2$. Indeed, otherwise we would have $\sum_{i=j}^\beta g_i\ge (\beta-j+1)9q^2\theta_2$ and Lemma \ref{delta} would show
\begin{align*}
(\beta-j+1)9q^2\theta_2+c_3(e_0-e_1)\le 2q^7+3q^6+4q^5+3q^4+2q^3+q^2.
\end{align*}
Using the lower bound for $\beta$ {, the fact that Lemma \ref{solid2} implies $c_3\le 5q^2$} as well as $e_0-e_1=q^5-q^4-5q^3+q$, we find that this is impossible. Hence $g_j<9q^2\theta_2$ and therefore $g_i<9q^2\theta_2$ for all $i\le j$. Now, let $Q_1$, $Q_2$ and $Q_3$ be three non-collinear points in $\{R_i:i\in\{j-q-1,\dots,j\}\}$ and let $M$ be the set of all points $R_i$ with $i\le j-q-2$ that do not lie in the plane $\pi:=\langle Q_1,Q_2,Q_3\rangle$. Then $|M|\ge \frac14q^3$. Also, each of the points $Q_i$ lies on less than $9q^2$ lines that meet $M$, since every such line lies in $\theta_2$ flags that are contained in the union of the $G_i$ with $i\le j-q-2$. Then Lemma \ref{solid1} gives a solid that contains at least $5q^2$ points of $M$. The statement follows now from Lemma \ref{solid2}.
\end{proof}

\textsc{Notation}. From now on we denote by $S$ the unique solid that contains all but at most $4q^2$ of the points $P_i$ with $i\in I$ and we use the following notation:
\begin{itemize}
\item $C_0:=\{F_i\mid i\in I,\ P_i\in S\}$.
\item $C_1:=\{F_i\mid i\in I,\ P_i\notin S\}$.
\item $C_2:=\{F_i\mid i\in\{1,\dots,\theta_3-q\}\setminus I, |F_i|=e_0\}$.
\item $C_3:=\{F_i\mid i\in\{1,\dots,\theta_3-q\}\setminus I, |F_i|<e_0\}$.
\item $c_i:=|C_i|$ for $i\in\{0,\dots,3\}$.
\item $W:=\{P\in S\mid P\not=P_i\forall i\in I\}$.
\item Let $M$ be the set of all line-plane flags $(l,\pi)$ for which $l\cap S$ is a point which lies in $W$.
\end{itemize}

\begin{lemma}\label{ci_properties}
We have
\begin{enumerate}[label=(\alph*)]
\item $C_0\cup C_1\cup C_2\cup C_3$ is a partition of $\F$.
\item $c_1+c_3\le4q^2$.
\item $|W|=\theta_3-c_0$.
\item Every point of $W$ lies on the plane of exactly $q^3\theta_2$ flags of $M$.\label{ci_properties-flagsonpofW}
\item $|M|=|W|q^3\theta_2$.\label{ci_properties-cardinalityofM}
\item $c_3\le 2q^2+6q$.\label{ci_properties_boundonc3}
\end{enumerate}
\end{lemma}
\begin{proof}
Statement (a) is obvious from the notation introduced above. The choice of $S$ together with Lemma \ref{solid3} implies statement (b). Since no two members of $\F$ of size $e_0$ have the same generic part, we have $|W|=|S\setminus C_0|=|S|-|C_0|=\theta_3-c_0$ and thus statement (c). Furthermore, each point $P\in W$ lies on $q^3$ lines that meet $S$ only in $P$ and each such line lies in $\theta_2$ planes. Hence, for every point $P\in W$, exactly $q^3\theta_2$ flags $(\ell,\pi)$ of $M$ satisfy $\ell\cap S=P$, which proves statements (d) and (e). Finally, statement (f)  follows from Lemma \ref{delta} and
\begin{align*}
    (2q^2+6q)(e_0-e_1)>2q^7+3q^6+4q^5+3q^4+2q^3+q^2.\tag*{\qedhere}
\end{align*}
\end{proof}

\begin{lemma}\label{HowCimeetsM}~
\begin{enumerate}[label=(\alph*)]
\item If $F\in C_0$, then the generic part of $F$ does not contain a flag of $M$.\label{HowCimeetsM_C1}\label{HowCimeetsM_C0}
\item If $F\in C_1$, then $|F\cap M|\le |W|\theta_2+q^2\theta_2$.
\item If $F\in C_2$, with base hyperplane $H$, then $|F\cap M|\le q^2\theta_2$, if $H=S$, and $|F\cap M|\le |H\cap W|q^2(q+1)+q^2\theta_2$ otherwise.\label{HowCimeetsM_C2}
\end{enumerate}
\end{lemma}
\begin{proof}
(a) The flags of the generic part of $F$ either have a line that is contained in $S$ or meets $S$ in the base point of $F$, which is not in $W$. Therefore these flags do not belong to $M$.

(b) We know that $F$ is based on a point $P$. The generic part of $F$ consist of all flags whose line contains $P$. As $P\notin S$, we see that the generic part of $F$ has exactly $|W|\theta_2$ flags in $M$. The special part of $F$ has $q^2\theta_2$ flags and thus at most this many flags of $M$.

(c) We know that $F$ is based on a hyperplane $H$. The generic part of $H$ consist of all flags whose plane lies in $H$. Hence, if $H=S$, the generic part contains no flag of $M$, and if $H\not=S$, it contains exactly $|H\cap W|q^2(q+1)$ flags of $M$. The special part of $F$ has $q^2\theta_2$ flags and thus at most this many flags of $M$.
\end{proof}

\begin{lemma}\label{Wglobal}
Suppose that $z$ is an integer such that all except at most one plane of $S$ have at most $z$ points in $W$. Then
\begin{align*}
|W|q^3\theta_2\le & c_1(|W|+q^2)\theta_2+c_2(zq^2(q+1)+q^2\theta_2)+c_3e_1+s+q^3(q+1)\theta_2,
\end{align*}
where $s$ is the number of flags of $M$ that are contained in the special part of $F$ for some EKR-set $F$ of $C_0$. If every plane of $S$ has at most $z$ points in $S$, then
\begin{align*}
|W|q^3\theta_2\le & c_1(|W|+q^2)\theta_2+c_2(zq^2(q+1)+q^2\theta_2)+c_3e_1+s.
\end{align*}

\end{lemma}
\begin{proof}
Each of the $|M|=|W|q^3\theta_2$ flags of $M$ is contained in some member of $\F=C_0\cup C_1\cup C_2\cup C_3$. Hence $|W|q^3\theta_2\le \sum_{i=0}^3\sum_{F\in C_i}|F\cap M|$. If there exists a plane of $S$ with more than $z$ points in $W$, then denote by $z'$ its number of points in $W$. Otherwise put $z':=z$. Since a plane of $S$ lies in $q$ solids other than $S$, the preceding lemma shows that $\cup_{F\in C_2}F$ and $M$ share at most
\begin{multline*}
(c_2-q)(zq^2(q+1)+q^2\theta_2)+q(z'q^2(q+1)+q^2\theta_2)\\=c_2(zq^2(q+1)+q^2\theta_2)+(z'-z)q^3(q+1)
\end{multline*}
flags. Since $|F|\le e_1$ for $F\in C_3$ and, using $z'-z\le\theta_2$ for the first assertion and $z'-z=0$ for the second assertion, they follow from the preceding lemma.
\end{proof}

\begin{lemma}\label{twoplanes}
Let $\pi_1$ and $\pi_2$ be distinct planes of $S$. Then
\begin{align}\label{eqnwithplane}
|(\pi_1\cup\pi_2)\cap W|q^2(q+1) \le 6q^3(q+3)+(|W|-q) 3q(q+1).
\end{align}
\end{lemma}
\begin{proof}
Put $W'=(\pi_1\cup\pi_2)\cap W$, and let $M'$ be the subset of $M$ that consists of all flags of $M$ whose line meets $S$ in a point of $W'$. Lemma \ref{ci_properties}~\ref{ci_properties-flagsonpofW} shows that $|M'|=|W'|q^3\theta_2$. Each flag of $M'$ lies in at least one of the EKR-sets of $\F=C_0\cup C_1\cup C_2\cup C_3$. Hence $|M'|\le d_0+d_1+d_2+d_3$ where $d_i$ is the number of elements of $M'$ that lie in some member of $C_i$.

For $F\in C_3$ we have $|F\cap M'|\le |F|\le e_1$. Hence $d_3\le c_3e_1$.

If $F\in C_1$, then $|F|=e_0$ and $F$ is based on a point $P \notin S$, so the flags of $M'$ that lie in the generic part of $F$ are precisely the $|W'|\theta_2$ flags whose line contains $P$ and a point of $W'$. Since the special part of $F$ has $q^2\theta_2$ flags, it follows that $d_1\le c_1(|W'|+q^2)\theta_2$.

Consider $F\in C_2$. Then $|F|=e_0$ and $F$ is based on a hyperplane $H$. If $H=S$, then the lines of all flags of $F$ are contained in $S$ and hence $F\cap M'=\emptyset$. Now we consider the case when $H\not=S$. Then the number of flags of $M'$ in the generic part of $F$ is $|H\cap W'|q^2(q+1)$. This number is at most $(2q+1)q^2(q+1)$, if the plane $H\cap S$ is different from $\pi_1$ and from $\pi_2$, and it is $|W\cap \pi_i|q^2(q+1)$, if $H\cap S=\pi_i$. Since there are exactly $q$ solids that meet $S$ in $\pi_1$ and as many that meet $S$ in $\pi_2$, it follows that the number of flags of $M'$ that lie in the generic part of at least one EKR-set of $C_2$ is at most
\begin{align*}
q(|W\cap \pi_1|+|W\cap \pi_2|)q^2(q+1)+(c_2-2q)(2q+1)q^2(q+1).
\end{align*}
The special part of each EKR-set of $C_2$ has $q^2\theta_2$ flags and thus at most this many flags of $M'$. Using $|W\cap \pi_1|+|W\cap \pi_2|\le|W'|+q+1$, it follows that
\begin{align*}
d_2\le q(|W'|+q+1)q^2(q+1)+(c_2-2q)(2q+1)q^2(q+1)+c_2q^2\theta_2.
\end{align*}
Finally we consider an EKR-set $F$ of $C_0$. Then $|F|=e_0$ and $F$ is based on a point $P$. We know from Lemma \ref{HowCimeetsM} that only the special part $T$ of $F$ can contribute to $M'$. For $T$ there are the following possibilities:
\begin{itemize}
\item There exists a line $\ell$ with $P\in \ell$ and $T$ consists of all flags whose plane contains $\ell$ and whose line does not contain $P$. If $\ell$ meets $S$ only in $P$, then $|T\cap M'|=|W'|q$. If $\ell$ is contained in $S$, then $|T\cap M'|=|\ell\cap W'|q^3$, which is at most $2q^3$, if $P\notin \pi_1\cup\pi_2$, and at most $q^4$, if $P\in\pi_1\cup\pi_2$. Since $|W'|\le 2q^2+q+1$, it follows that $|T\cap M'|\le q^4$, if $P\in \pi_1\cup\pi_2$, and $|T\cap M'|\le q(2q^2+q+1)$ otherwise.
\item There exists a solid $H$ with $P\in H$ and $T$ consists of all line-plane flags $(h,\tau)$ with $P \in  \tau\subseteq H$ and $P\notin h$. Then $T\cap M'=\emptyset$, if $H=S$, and $|T\cap M'|=|H\cap W'|q^2$, if $H\not=S$. In the second case this number is $|W'\cap\pi_i|q^2$ if $H\cap S=\pi_i$ for some $i\in\{1,2\}$, and it is at most $(2q+1)q^2$, if $H\cap S\notin\{\pi_1,\pi_2\}$. Notice that $H\cap S=\pi_i$ implies $P\in\pi_i$, so that $|W'\cap\pi_i|\le q^2+q$ and hence $|W'\cap\pi_i|q^2\le q^3(q+1)$.
\end{itemize}

Summarizing we see that $|T\cap M'|\le q(2q^2+q+1)$, if $P\notin \pi_1\cup\pi_2$, and $|T\cap M'|\le q^3(q+1)$, if $P\in \pi_1\cup\pi_2$, which proves
\begin{align*}
    d_0
        &\le (c_0-2q^2-q-1+|W'|)q(2q^2+q+1)+(2q^2+q+1-|W'|)q^3(q+1)\\
        &\le c_0q(2q^2+q+1)+2q^6-q^5-|W'|(q^4-q^3-q^2-q).
\end{align*}
It follows that
\begin{align}
|W'|q^3\theta_2
    &=|M'|\le d_0+d_1+d_2+d_3\nonumber\\
    &\le c_0q(2q^2+q+1)+2q^6-q^5-|W'|(q^4-q^3-q^2-q)\nonumber\\
    &\hphantom{\le{}}\mathrel{+}c_1(|W'|+q^2)\theta_2+q(|W'|+q+1)q^2(q+1)\nonumber\\
    &\hphantom{\le{}}\mathrel{+}(c_2-2q)(2q+1)q^2(q+1)+c_2q^2\theta_2+c_3e_1\nonumber
\intertext{and simplifications show}
|W'|q^4\theta_1
    &\le |W'|q\theta_2+q^3(2q^3-4q^2-4q-1)+c_0q(q^2+\theta_2)\nonumber\\
    &\hphantom{\le{}}\mathrel{+}\underbrace{c_1(|W'|+q^2)\theta_2+ c_2q^2(\theta_2+2q^2+3q+1)+c_3e_1}_{=:\xi}.\label{eqndelta}
\end{align}
We put $\delta:=c_1+c_2+c_3$, which also implies $c_0=\theta_3-q-\delta$, and use $|W'|\le q^2+\theta_2$ as well as $c_3\le 2q^2+6q$ from Lemma \ref{ci_properties}~\ref{ci_properties_boundonc3}, to see that
\begin{align*}
\xi
    &\le\delta(3q^2+q+1)\theta_2+c_3(e_1-(3q^2+q+1)\theta_2)\\
    &\le\delta(3q^2+q+1)\theta_2+(2q^2+6q)(q^4+5q^3-q^2-q).
\end{align*}
Using this as well as $|W'|\le q^2+\theta_2$ and $c_0=\theta_3-q-\delta$ on the right hand side of
inequality (\ref{eqndelta}) we find
\begin{align*}
|W'|q^4\theta_1 \le
6q^6+17q^5+29q^4-2q^3-3q^2+2q+\delta (3q^4+2q^3+4q^2+q+1).
\end{align*}
Substituting $\delta=|W|-q$ therein implies the statement.
\end{proof}

\begin{lemma}\label{Watmostquadraticinq}
We have $c_0\ge q^3-18q+1$ and thus $|W|\le q^2+19q$.
\end{lemma}
\begin{proof}
Let $\pi_1$ and $\pi_2$ be planes of $S$ such that $|\pi_1\cap W|\ge|\pi_2\cap W|\ge|\pi\cap W|$ for every plane $\pi$ of $S$ other than $\pi_1$. Put $z=|\pi_2\cap W|$.
The number $s$ occurring in the assertion of Lemma \ref{Wglobal} is at most $c_0q^2\theta_2$, since the special part of each EKR-sets of $C_0$ has cardinality $q^2\theta_2$. Therefore, Lemma \ref{Wglobal} shows \begin{align*}
|W|(q^3-c_1)\theta_2\le c_0q^2\theta_2+c_1q^2\theta_2+c_2(zq^2(q+1)+q^2\theta_2)+c_3e_1+q^3(q+1)\theta_2.
\end{align*}
Since $c_0+c_1+c_2+c_3=\theta_3-q$, the right hand side is equal to
\begin{align*}
(\theta_3-q)q^2\theta_2+c_2zq^2(q+1)+c_3(e_1-q^2\theta_2)+q^3(q+1)\theta_2.
\end{align*}
Using $c_3\le 2q^2+6q$ from Lemma \ref{ci_properties}~\ref{ci_properties_boundonc3} and the definition of $e_1$ implies
\begin{align*}
|W|(q^3-c_1)\theta_2\le q^7+10q^6+c_2zq^2(q+1).
\end{align*}
We put $\delta:=c_1+c_2+c_3$, such that $|W|=\theta_3-c_0=\delta+q$ and thus
\begin{align}\label{Wgrob1}
(\delta+q)(q^3-c_1)\theta_2\le q^7+10q^6+\delta zq^2(q+1).
\end{align}
Now, the preceding lemma states
\begin{align*}
    |(\pi_1\cup\pi_2)\cap W|q^2(q+1) \le 6q^4+18q^3+3\delta(q^2+q)
\end{align*}
and, since $|(\pi_1\cup\pi_2)\cap W|\ge|\pi_1\cap W|+|\pi_2\cap W|-(q+1)\ge 2z-q-1$, this implies
\begin{align}\label{Wgrob2}
    2zq^2(q+1) \le 8q^4+3\delta (q^2+q).
\end{align}
Combining (\ref{Wgrob1}) with (\ref{Wgrob2}) and using $c_1\le 4q^2$ results in
\begin{align*}
(\delta+q)(q^3-4q^2)\theta_2\le\ & q^7+10q^6+\delta(4q^4+\frac32\delta (q^2+q)).
\end{align*}
It is easy to verify that this inequality is not satisfied for $\delta=q^2+18q$ nor for $\delta=\frac23q^3-7q^2$. Since $\delta$ occurs quadratic in the inequality, it follows that $\delta$ does not lie in the interval $[q^2+18q,\frac23q^3-7q^2]$.
However, we have $\delta=\theta_3-q-c_0$ as well as $c_0+c_1=|I|\ge \frac{1}{2}(\theta_3-q-c_0)$ and since $c_1+c_3\le 4q^2$ this implies $\delta<\frac23q^3-7q^2$, that is, $\delta\le q^2+18q$.
\end{proof}

\begin{lemma}\label{planes_are_small}
Every plane of $S$ has at most $10q$ points in $W$.
\end{lemma}
\begin{proof}
This is an implication of Lemma \ref{twoplanes} and $|W|\le q^2+19q$ from Lemma \ref{Watmostquadraticinq}.
\end{proof}

\begin{theorem}
We have $\F=C_0$.
\end{theorem}
\begin{proof}
Like in the previous proofs we put $\delta:=c_1+c_2+c_3$, which again implies $|W|=q+\delta$ as well as $\delta=\theta_3-q-c_0$. From \ref{Watmostquadraticinq} and \ref{planes_are_small} we have $|W|\le q^2+19 q$ and $|\pi\cap W|\le 10q$ for all planes $\pi$ of $S$. Therefore Lemma \ref{HowCimeetsM} shows that each of the EKR-sets $F\in C_1\cup C_2$ satisfies $|F\cap M|\le12q^4$. Since $e_1<12q^4$, the same holds for $F\in C_3$. Therefore, the total contribution of all EKR-sets in $C_1\cup C_2\cup C_3$ to $M$ is at most $12\delta q^4=12(|W|-q)q^4$. The generic part of the EKR-sets in $C_0$ are disjoint from $M$. Furthermore, the generic part of every EKR-set in $C_0$ is disjoined from $M$ and thus it remains to consider the special parts $T(F)$ of the EKR-sets $F\in C_0$. In view of that we define
\begin{align*}
    \gamma
        &:=|\{F\in C_0:T(F)\text{ is based on a line }\ell\not\le S\}|,\\
    \alpha
        &:=|\{F\in C_0:T(F)\text{ is based in a line }\ell\le S\}|,\\
    \beta
        &:=|\{F\in C_0:T(F)\text{ is based on a hyperplane }H\}|.
\end{align*}
Moreover, we let $A$ be the set of lines $\ell$ of $S$ such that $\FFF(P,\ell)\in C_0$ for some point $P$ of $\ell$ and we let $B$ be the set of all point-hyperplane pairs $(P,H)$ with $\FFF(P,H)\in C_0$ and $H\not=S$. Then $\alpha+\beta+\gamma=c_0$, $|A|\le\alpha$ and $|B|\le\beta$. Recall that Lemma \ref{HowCimeetsM}~\ref{HowCimeetsM_C2} shows that, if $F\in C_0$ is such that $T(F)$ is hyperplane based with hyperplane $S$, then $T(F)$ does not contribute to $M$. Therefore, we have
\begin{align}\label{eqnsmallWW}
    |W|q^3\theta_2\le12(|W|-q)q^4+\sum_{\ell\in A}|\ell\cap W|q^3+\sum_{\mathclap{(P,H)\in B}}|H\cap W|q^2+\gamma |W|q.
\end{align}
Furthermore, since the product of two consecutive integers is non-negative we have
\begin{align*}
0
    &\le\sum_{\ell\in A}(|\ell\cap W|-1)(|\ell\cap W|-2)\\
    &=\sum_{\ell\in A}|\ell\cap W|(|\ell\cap W|-1)-2\sum_{\ell\in A}|\ell\cap W|+2|A|\\
    &\le|W|(|W|-1)-2\sum_{\ell\in A}|\ell\cap W|+2|A|.
\end{align*}
Since $\alpha+\beta+\gamma=\theta_3-|W|$ and $|A|\le\alpha$ we have $|A|\le\theta_3-|W|-\beta-\gamma$ and thus this equation implies
\begin{align*}
    \sum_{\ell\in A}|\ell\cap W|\le \frac{1}{2}|W|(|W|-3)+\theta_3-\beta-\gamma.
\end{align*}
Using this and $|B|\le\beta$ in (\ref{eqnsmallWW}) we find
\begin{align}\label{stillwithgamma}
    L:=|W|q^3\left(\theta_2-\frac{1}{2}(|W|-3)\right)
        &\le 12(|W|-q) q^4+(\theta_3-\gamma)q^3\nonumber\\
        &\hphantom{\le{}}\mathrel{+}\gamma |W|q+\sum_{\mathclap{(P,H)\in B}}(|H\cap W|-q)q^2.
\end{align}
Now, we first show that the coefficient of $\gamma$ in this equation is negative, so that we may omit $\gamma$ therein. Since $|W|\le q^2+19q$ we have $L\ge\frac{1}{3}|W|q^5$. Furthermore, Lemma \ref{planes_are_small} shows $|H\cap W|\le 10q$ for all $(P,H)\in B$ and, since $|B|+\gamma\le\beta+\gamma\le \theta_3-|W|$, we find that
\begin{align*}
    \gamma|W|q+\sum_{\mathclap{(P,H)\in B}}(|H\cap W|-q)q^2\leq \gamma(q^3+19q^2)+9q^3|B|\le (\theta_3-|W|)9q^3.    
\end{align*}
Using this as well as $|W|\le q^2+19q$ and $L\ge\frac{1}{3}|W|q^5$ in Equation (\ref{stillwithgamma}) we have
%\footnote{\textcolor{magenta}{This is maybe not clear enough, since if we use $\frac{1}{3}Wq^5\leq 12(q^2+18q)q^4+\theta_3q^3+(\theta_3-W)9q^3$, then we already have that $W$ has order $q$. And is it true that we bound $\gamma|W|q+ \sum_{\mathclap{(P,H)\in B}}(|H\cap W|-q)q^2$, only to prove that the coefficient of $\gamma$ is negative?\textcolor{blue}{it seems like it}}}
\begin{align*}
    \frac{1}{3}|W|q^5\le12(q^2+18q)q^4+\theta_3q^3+(\theta_3-|W|)9q^3
\end{align*}
and thus $|W|\le\frac{1}{2}q^2$. Hence, the coefficient $|W|q-q^3$ of $\gamma$ in (\ref{stillwithgamma}) is negative and therefore $\gamma$ can be omitted in the inequality. Doing that, replacing $|W|$ by $q+\delta$ and simplifying we find
\begin{align}\label{eqnprofinal}
    (q+\delta)q\left(q^2+\frac{1}{2}(q-\delta+5)\right)&\le12\delta q^2+\theta_3q+\sum_{\mathclap{(P,H)\in B}}(|H\cap W|-q).
\end{align}
If $\pi$ is a plane of $S$, then the number of $(P,H)\in B$ with $H\cap S=\pi$ is at most $\theta_2-|\pi\cap W|$. Also, if $\pi_1$ and $\pi_2$ are distinct planes of $S$, then
\begin{align}\label{blapi1pi2}
  |\pi_1\cap W|+|\pi_2\cap W|\le |W|+|\pi_1\cap \pi_2\cap W|\le 2q+1+\delta.  
\end{align}
We claim that
\begin{align}\label{eqnsumestimate}
    \sum_{\mathclap{(P,H)\in B}}(|H\cap W|-q)\le\frac{1}{2}(\theta_3-q-\delta)(\delta+1).
\end{align}
Since $|B|\le\theta_3-q-\delta$, this is clear if $|H\cap W|\le q+\frac{1}{2}(\delta+1)$ for all $H\in B$. Hence, we may assume that there exists a flag $(P_0,H_0)\in B$ with $|H_0\cap W|=q+x$ and $x\ge\frac{1}{2}(\delta+1)$. Then Equation (\ref{blapi1pi2}) implies $|H\cap W|\le q+1+\delta-x$ for all $(P,H)\in B$ with $H\cap S\not=H_0\cap S$. Furthermore, if $(P,H)\in B$ with $H\cap S=H_0\cap S$, then $P\in H_0\cap S$ and hence this happens at most $\theta_2-q-x$ times. Together this implies
\begin{align*}
\sum_{(P,H)\in B}(|H\cap W|-q)
    &\le(\theta_2-q-x)x+(\theta_3-q-\delta-(\theta_2-q-x))(\delta+1-x)\\
    &=\frac{1}{2}((\theta_3-q-\delta)(\delta+1)-\xi(\xi+q^3-q^2)),
\end{align*}
where $\xi:=(2x-\delta-1)>0$ and thus (\ref{eqnsumestimate}) holds either way.

Now, we may use the bound (\ref{eqnsumestimate}) in Equation (\ref{eqnprofinal}) to find
\begin{align*}
    (q+\delta)q\left(q^2+\frac{1}{2}(q-\delta+5)\right)
        &\le12\delta q^2+\theta_3q+\frac{1}{2}(\theta_3-q-\delta)(\delta+1),
\end{align*}
which is equivalent to
\begin{align}
    \frac{1}{2}\delta q(q^2-25q-\delta+5)+\frac{1}{2}\delta^2\le(q^2-q+1)q+\frac{1}{2}.
\end{align}
Using $\delta=|W|-q<\frac{1}{2}q^2$, this implies $\delta<3$ and thus $\delta\le 2$, that is, it remains to show $\delta\notin\{1,2\}$.

First consider $\delta=2$. Then Equation (\ref{eqnprofinal}) shows
\begin{align*}
    \frac{1}{2}(3q^3-45q^2+4q)\le\sum_{\mathclap{(P,H)\in B}}(|H\cap W|-q).
\end{align*}
Since $|B|\le c_0=\theta_3-q-\delta$ and since $|H\cap W|\le|W|=q+2$ for all $(P,H)\in B$, this implies $|H\cap W|>q+1$ and thus $|H\cap W|=q+2=|W|$ for at least $\frac{1}{2}(q^3-47q^2+4q+2)$ elements $(P,H)\in B$. Notice that $|H\cap W|=q+2$ implies $W\subseteq H$, that is, $W\subseteq H\cap S$. Therefore $W$ spans a plane $\sigma$ of $S$. However, $(P,H)\in B$ with $W\subseteq H$ implies $P\in H\cap S=\sigma$ and this may happen at most $\theta_2-|W|=q^2-1$ times, a contradiction.

Now, suppose that $\delta=1$. Then Equation (\ref{eqnprofinal}) shows
\begin{align*}
    \frac{1}{2}(q^3-21q^2+2q)\le\sum_{\mathclap{(P,H)\in B}}(|H\cap W|-q)
\end{align*}
and, since $|H\cap W|\le|W|=q+1$ for all $(P,H)\in B$, this implies that there are $\frac{1}{2}(q^3-21q^2+2q)$ elements $(P,H)\in B$ with $|H\cap W|>q$ and thus $|H\cap W|=q+1$. Now, if $W$ spans a plane $\sigma$ then we have seen above that there are at most $\theta_2-|W|=q^2$ elements $(P,H)\in B$ with $W\le H$, a contradiction. Therefore, we may assume that $W$ spans a line $\ell$, only. Hence, finally, there exists only one EKR-set $F$ in $\F\setminus C_0$. Now, the special parts of the EKR-sets of $C_0$ do not contain any flag $(h,\pi)$ with $\pi\cap S=\ell$ and therefore these $q^2\theta_2$ flags must lie in $F$. This implies that $F$ may not be a subset of a hyperplane-based EKR-set, nor may it be a subset of a point based EKR-set with point outside of $S$. Hence, we have $|F|\le e_1$.

Now, reconsider the set $M$ of all $|W|q^3\theta_2=(q+1)q^3\theta_2$ flags $(h,\tau)$ such that $h\cap S$ is a point of $W$. Each point $P\in S\setminus W$ is the base point of exactly one EKR-set of $C_0$ and we let $S(P)$ be its special part. Then $M$ is a subset of the union of $F$ and the sets $S(P)$ with $P\in S\setminus W$. The $q^2(q+1)$ points of $S\setminus W$ are distributed in the $q+1$ planes of $S$ on $\ell$. Consider such a plane $\pi$ and let
\begin{enumerate}[label=$\bullet$]
    \item
        $\gamma_\pi$ be the number of points $P\in\pi\setminus\ell$ for which $S(P)$ is based on a line that meets $S$ only in $P$,
    \item
        let $a_\pi$ be the number of points $P\in\pi\setminus\ell$ for which $S(P)$ is based on a line that is contained in $S$, and
    \item
        let $b_\pi$ be the number of pairs $(P,H)\in B$ with $P\in \pi$.
\end{enumerate}
Then there are at most $\gamma_\pi(q+1)q+a_\pi q^3$ flags in $M$ that lie in $S(P)$ for some point $P\in\pi\setminus\ell$ such that $S(P)$ is based on a line. Now, consider the $b_\pi$ pairs $(P,H)\in B$ with $P\in \pi$. The special part $S(P)$ of every such pair contains $|H\cap\ell|q^2$ pairs of $M$. If $\ell\not\subseteq H$, then this is $q^2$ and otherwise it is $q^2(q+1)$. For distinct $(P_1,H_1),(P_2,H_2)\in B$ with $P_1,P_2\in\pi$ and $\pi\subseteq H_1=H_2$, the $q^2$ flags $(g,\tau)\in M$ for which $\tau\cap S=P_1P_2$ (and hence $g\cap S=P_1P_2\cap g$) lie in both $S(P_1,H_1)$ and $S(P_2,H_2)$, so that the number of flags of $M$ that lie in $S(P_2)$ but not in $S(P_1)$ is at most $q^3$. Since there are $q$ hyperplanes on $\pi$ different from $S$, these arguments show that the union of the special parts $S(P)$ for the $b_\pi$ points in question is at most $q\cdot (q+1)q^2+(b_\pi-q)q^3=(b_\pi+1)q^3$. Therefore, the union of the special parts $S(P)$ for all points $P\in \pi\setminus\ell$ contains at most
\[
    \gamma_\pi(q+1)q+a_\pi q^3+(b_\pi+1)q^3\le (\gamma_\pi+a_\pi+b_\pi)q^3+q^3=(q^2+1)q^3.
\]
Since there are $q+1$ planes of $S$ on $\ell$, it follows that
\[
    (q+1)q^3\theta_2\le |F|+(q+1)(q^2+1)q^3
\]
which shows $|F|\ge (q+1)q^4$. This is a contradiction to $|F|\le e_1$.
\end{proof}

\begin{corollary}
The chromatic number of $qK_{5;\{2,3\}}$ is $\theta_3-q$.
\end{corollary}

\textsc{remark}: We have also shown that every coloring class of a coloring with $q^3+q^2+1$ colors is a subset of a maximal independent set of $qK_{5;\{2,3\}}$, that is an independent set of cardinality $e_0$.

\begin{appendix}
\section{Appendix}
    {In \cite{blokhuis1}, the authors investigate EKR-sets of line-plane flags in $\PG(4,q)$. In the proof of their classification result, they consider EKR-sets of line-plane flags in $\PG(4,q)$ which are not contained in one of the sets given in Example \ref{voorbeeld}. For this, the authors distinguish several cases for the structure of such a set and, in most cases, provide an upper bound for its size. The weakest of these upper bounds is the number $e_1$ and it is given in the general case of the proof of Lemma 4.3 of \cite{blokhuis1}.}
    %(The proof of \cite{blokhuis1} is such that it considers EKR-sets of line-plane flags in $\PG(4,q)$ which are not a subset of one of the sets given in Example \ref{voorbeeld}. In doing so the authors consider several different cases that may occur for the structure of such a set and, in most cases, provide an upper bound for its size. The weakest of these upper bounds is the number $e_1$ and it is given in the general case of the proof of Lemma {This has to be a reference: 4.3} of \cite{blokhuis1}.)

    However, in Case D of Section 4.1 they do not actually provide an upper bound, but only show {that in this case, the sets cannot be contained in a set of Example \ref{voorbeeld}.}
    %(that the number of flags in the set they study is smaller than\footnote{\textcolor{magenta}{Here it must be $e_0$, or the pink sentence, I think.}} $e_1$.)
    In order to use their result we first have to provide an upper bound for that case, too, and we shall do so below. We adapt their notation for the remainder of the appendix.

We are in the situation that there is one red plane $A_0$ and all of its lines are red as well. If there are more than $q+1$ red planes, then the arguments in \cite{blokhuis1} (second paragraph of section D) show that the number of elements in the EKR set is at most $e_1$. So here we consider the case that there are at most $q$ red planes apart from $A_0$.

If $A$ is a yellow plane, then $A\cap A_0$ is a line (so $A_0+A$ is a solid) and $A$ has a unique point $p(A)$ which lies in $A_0$ and such that a flag $(L,A)$ is in $\cal C$ if and only if $p(A)\in L$. The following holds and will be used several times below: If $A_1$ and $A_2$ are yellow planes, then $p(A_1)\in A_2$ or $p(A_2)\in A_1$ or $A_0+A_1=A_0+A_2$.

Suppose for any two yellow planes $A_1$ and $A_2$ with $p(A_1)=p(A_2)$ we have
%\footnote{\textcolor{magenta}{I think this must be $A_0\cap A_1=A_0\cap A_2$}} 
$A_0\cap A_1=A_0\cap A_2$.
Then each point $P$ of $A_0$ satisfies $P=p(A)$ for at most $q^2+q$ yellow planes. In this situation, there are at most $\theta_2(q^2+q)$ yellow planes.

Suppose now that there is a point $P$ and two yellow planes $A_1$ and $A_2$ with
%\footnote{\textcolor{magenta}{I think this must be $A_0\cap A_1\neq  A_0\cap A_2$}}
$A_0\cap A_1\not=A_0\cap A_2$ and $p(A_1)=p(A_2)=P$. Then each yellow plane must satisfy $P\in A$ or $A\subseteq A_0+A_1$ or $A\subseteq A_0+A_2$. The number of yellow planes is thus at most $2(q^3+q^2+q)+(q+1)(q^2-q)$. Note that equality can occur only when the solids $\langle A_0, A_1\rangle$ and $\langle A_0, A_2\rangle$ are
%\footnote{\textcolor{magenta}{I think this must be the span in stead of the intersection?}} 
distinct.

In any case, the number of yellow planes is at most $y:=3q^3+2q^2+q$. If $A_0$ is the only red plane, it follows that $|{\cal C}|\le \theta_2^2+yq\le 4q^4+4q^3+4q^2+2q+1$. If $A_0$ is not the only red plane and there are $q$ others red planes $A$, then we treat these as the yellow planes above by choosing for $p(A)$ any point of $A\cap A_0$. Then the bound for $|{\cal C}|$ is almost the same except that we have to add $q\cdot q^2$, namely $q^2$ more flags for each of the $q$ red planes. Hence $|{\cal C}|\le e_1$.
\end{appendix}

\end{document}